\documentclass[11pt]{amsart}

\usepackage[utf8]{inputenc}

\usepackage{a4wide}

\usepackage{amssymb}
\usepackage{amsmath}
\usepackage{amsthm}
\usepackage{amscd}
\usepackage{mathtools}

\usepackage{diagrams}
\newarrow{To}----{->}

\newtheorem{theorem}{Theorem}[section]
\newtheorem{lemma}[theorem]{Lemma}
\newtheorem{proposition}[theorem]{Proposition}
\newtheorem{corollary}[theorem]{Corollary}

\newtheorem{remark}[theorem]{Remark}

\newcommand{\pandzic}{Pand$\check{\mathrm{z}}$i\`c~}

\newcommand{\dolbeault}{\overline{\partial}}
\usepackage{color}

\begin{document}

\baselineskip=18pt
\sloppy
\title[Translation of Dolbeault representations]
{Translation of Dolbeault representations
on reductive homogeneous spaces}
\author{N. Prudhon}
\thanks{{\em 2000 Mathematics Subject Classification}. Primary 22E46; Secondary 22F30}
\keywords{Dolbeault cohomology; Cohomological induction; Translation functors}
\address{Universit\'e de Lorraine, Institut Elie Cartan de Lorraine, UMR 7502 - CNRS, France}
\email{nicolas.prudhon@univ-lorraine.fr}

\begin{abstract}
We adapt techniques used in the study of the cubic 
Dirac operator on homogeneous reductive spaces to the Dolbeault 
operator on elliptic coadjoint orbits to prove that cohomologically  induced 
representations have an infinitesimal character, that
cohomological induction and Zuckerman translation 
functor commute and give a geometric interpretation
of the Zuckerman translation functor in this context.
\end{abstract}

\maketitle

\section*{Introduction}
In their proof of a conjecture of Vogan on
Dirac cohomology for semi simple Lie groups 
Huang and \pandzic \cite{HuangPandzic2002} 
introduced a differential
complex whose differential is given by
the commutator with the algebraic Dirac operator.
The cohomology of this complex is computed using
its commutative analogue given by the symbol map.
Later Alekseev and Meinrenken \cite{AM00,AM05} gave
an interpretation of their computation in terms
of the non commutative Weil algebra and the Chern-Weil homomorphism that eventually leads to an easy proof 
of the Duflo theorem for quadratic Lie algebras
as well as a theorem of Rouviere for symmetric pairs.
Here we use these powerful techniques with the
Dolbeault operator instead and recover easily some 
already known results on representation theory.
In particular we show  that Dolbeault cohomology
representations have an infinitesimal
character, that the Zuckerman translation functor
is well defined in this geometric context and commute
with Dolbeault cohomology induction. Moreover we give
a simple geometric interpretation of the Zuckerman
translation functor for these modules. 
In fact it appears to be a projection in the fibers
of the vector bundles involved. This also includes
a proof of a theorem of Casselman and Osborne
as well as results of Kostant on 
$\mathfrak{u}$-cohomology. Similar proofs of Casselman-Osborne
theorem also appear in \cite{HuangPandzicRenard-arxiv} 
and an unpublished work of M. Duflo \cite{duflo}.

Let $G$ be a connected real reductive Lie group with complexified 
Lie algebra $\mathfrak{g}$. 
By reductive we mean that $\mathfrak{g}$ decomposes as 
$[\mathfrak{g},\mathfrak{g}]+\mathfrak{z}$, 
where $\mathfrak{z}$ denotes the 
center of  $\mathfrak{g}$. If $K/Z$ is a maximal compact subgroup of 
$G/Z$, where $Z$ is the center of $G$, then $K$ is the fixed 
point group of a Cartan involution $\theta$ of $G$. 
Write $\mathfrak{g}=\mathfrak{k}+\mathfrak{s}$ for 
the corresponding Cartan decomposition of 
$\mathfrak{g}$, where $\mathfrak{k}$ is the 
complexified Lie algebra of $K$. Let $G^\mathbb{C}$,
the complexified Lie group of $G$. We have
$\mathrm{Lie}(G^\mathbb{C})=\mathfrak{g}$ 
and we will use the same convention for various
subgroups of $G$. We fix a $\mathrm{Ad}(G)$-invariant 
bilinear form $B$ on $\mathfrak{g}$ that coincides
with the Killing form on 
$[\mathfrak{g},\mathfrak{g}]$.

If $H$ is a closed connected reductive $\theta$-subgroup of $G$ and $(V,\tau)$  a finite dimensional representation of $H$, we then have a  finite rank complex vector bundle
$\mathcal{V}=G\times_H V$ over $G/H$. We will consider 
smooth sections of such a vector bundle.
This space $\Gamma(V)$ of smooth sections 
is identified with the spaces $(C^\infty(G)\otimes V)^H$ or $C^\infty(G,\,V)^H$.
If $V$ is a smooth infinite dimensional representation of $H$, this last space
is still well defined, and the tensor product in the first space stands for the projective tensor product. They are again canonically isomorphic and we continue 
to call them the space of sections of the associated infinite dimensional bundle.
The homogeneous spaces we will consider in the sequel are elliptic coadjoint orbits. 
These spaces may be realized as measurable open $G$-orbits $Y=G/H$
in the complex flag manifold $Z=G^\mathbb{C}/Q$ of the 
complexified Lie group $G^\mathbb{C}$. 
The definition of an open measurable $G$-orbit is given by Wolf in \cite{wolf}
as well as the following facts.	A base point $z_0\in Y$ may be chosen so that 
$Q=\mathrm{Stab}_{G^\mathbb{C}}(z_0)$, 
and $H=Q\cap \overline{Q}=\mathrm{Stab}_G(z_0)$ 
contains a fundamental Cartan $T$ 
subgroup of $G$. One 
may also assume that $\mathfrak{h}$ is the centralizer of a compact torus 
$\mathfrak{t}'\subset \mathfrak{t}\cap \mathfrak{k}$. 
So there exists $\xi_0\in \mathfrak{t}'$ such that $\mathrm{ad} \xi_0$ has
real eigenvalues, $\mathfrak{h}$ is the centralizer of $\xi_0$ and
$\mathfrak{u}$ is the sum of the eigenspaces of 
$\mathrm{ad}\xi_0$ corresponding to the positive eigenvalues, and
$\mathfrak{q}=\mathfrak{h}\oplus \mathfrak{u}$.
%
It can then be shown that the homogeneous
space $Y=G/H$ is isomorphic to a coadjoint elliptic orbit
and that all coadjoint elliptic orbit arise in this way.  

As an open submanifold of a complex manifold,
$Y$ is a complex manifold as well with a $G$-invariant complex structure.
With these assumptions the antiholomorphic tangent space at the origin
is $\mathfrak{u}$. 
Moreover $\mathfrak{u}$ and $\overline{\mathfrak{u}}$ are
isotropic subspaces in perfect duality for $B$. 
So we will
identify $\mathfrak{u}^*$ and 
$\overline{\mathfrak{u}}$ occasionally.
Let us now define a twisted Dolbeault
operator. It is a $G$-invariant differential operator
on the bundle $G\times_H \wedge^\bullet\mathfrak{u}^*\otimes E$
where $E$ is a smooth representation of $H$. 
Let $X_i$ be a basis of $\mathfrak{u}$, $\xi_i$ the dual 
basis. We note $r(X_i)$ the left invariant vector field on 
$G$ defined by right derivative and $e(\xi_i)$ 
(resp. $\iota(X_i)$)
the exterior product
(resp. interior product)
on $\wedge^\bullet\mathfrak{u}^*$.
The twisted Dolbeault operator is defined by
\begin{equation} \label{dolbeault}
 \dolbeault(E)=\sum_i r(X_i)\otimes e(\xi_i) \otimes I_E
 +\sum_{i<j} 1\otimes e(\xi_i)e(\xi_j)\iota([X_i,X_j])\otimes I_E
\end{equation}
It defines a differential complex on the bundle 
$G\times_H \wedge^\bullet\mathfrak{u}^*\otimes E$
whose cohomology $H_{\dolbeault(E)}$ is known as 
Dolbeault cohomology. 
Again when $E$ is any smooth representation of $H$,
the operator in equation~(\ref{dolbeault}) is well defined. 
H. W. Wong proved \cite{wong:95, wong:99} that the 
cohomology of the Dolbeault operator is a smooth admissible $G$-module
that is a maximal globalization of its underlying
Harish-Chandra module namely the cohomologically induced
Zuckerman module $\mathcal{R}^*(E)$. The main step is to prove that
the operator $\dolbeault(E)$ has closed range, 
which is a deep and difficult
result. Once this is done the results
on the cohomologically induced
Zuckerman module $\mathcal{R}^*(E)$ apply. In particular
Wong deduce that when $E$ has infinitesimal character
then the cohomology space
$H_{\dolbeault(E)}$ also has an infinitesimal
character. One  may also probably deduce the compatibility
of Dolbeault cohomology with the Zuckerman
translation functor from Wong's result. 
The main goal of this paper is to recover (resp. prove) these results 
on infinitesimal character and translations without the knowledge that 
the differential operator $\dolbeault(E)$ has the closed
range property.

More precisely we obtain the following results.
The representation of $G$
on $(C^\infty(G)\otimes \wedge^\bullet \mathfrak{u}^*\otimes E)^H$
commute with the Dolbeault operator. Hence its derivative
also commute with the Dolbeault operator and goes down
to a well defined representation of the enveloping algebra 
$\mathcal{U}(\mathfrak{g})$. Its restriction to the center
$\mathcal{Z}(\mathfrak{g})$ of $\mathcal{U}(\mathfrak{g})$
is given by the following theorem.

\begin{theorem} \label{mainthm}
Let $E$ be a smooth representation of $H$ which is
$\mathcal{Z}(\mathfrak{h})$ finite. Let us write
$E=\oplus_\mu E(\mu)$ for its decomposition in primary 
component such that $E(\mu)$ has generalized 
infinitesimal character $\chi_\mu$. Assume that each
$E(\mu)$ is a smooth $H$-module.
Then the $G$-module $H_{\dolbeault(E(\mu))}$ has a
generalized infinitesimal character with
Harish Chandra parameter $\mu+\rho(\mathfrak{u})$
and $H_{\dolbeault(E)}=\oplus_\mu H_{\dolbeault(E(\mu))}$.
\end{theorem}
Let $E(\mu)$ be a smooth representation of $H$ with generalized 
infinitesimal character $\chi_\mu$. 
Let $F^\nu$ be a finite dimensional irreducible 
representation of $G$ with highest weight $\nu$.
Its restriction to the connected group $H$ has a 
subrepresentation isomorphic to the representation $E^\nu$.
Hence we get a homomorphism of smooth representations of $H$
$$
p^H_{\mu,\nu} \colon E(\mu) \otimes F^\nu 
\to E(\mu)\otimes E^\nu\,.
$$
This projection induces a $G$-equivariant projection
$$
p^G_{\mu,\nu}\colon
C^\infty(G,\wedge^\bullet {\mathfrak{u}}^*
\otimes E(\mu))^H
\otimes F^\nu \to 
C^\infty(G,\wedge^\bullet {\mathfrak{u}}^*
\otimes E(\mu)\otimes E^\nu))^H
$$ 

\begin{theorem} \label{mainthm-2}
The map $p^G_{\mu,\nu}$
that goes down to cohomology
$$ 
p^G_{\mu,\nu}\colon H_{\dolbeault(E(\mu))}\otimes F^\nu
\to 
H_{\dolbeault(E(\mu)\otimes E^\nu))}\,.
$$
\end{theorem}

Thanks to theorem \ref{mainthm} and again the above 
theorem of Kostant, the representation 
$H_{\dolbeault(E(\mu))}\otimes F^\nu$ is a $\mathcal{Z}(\mathfrak{g})$
finite $\mathcal{U}(\mathfrak{g})$ module. We look at its primary 
component with generalized infinitesimal character 
$\chi_{\mu+\nu+\rho(\mathfrak{u})}$.

\begin{theorem} \label{mainthm-3}
Assume that 
\begin{equation*}\tag{C}
\mu + \rho(\mathfrak{u})+\nu 
\text{ is at least singular as }
\mu + \rho(\mathfrak{u})\,.
\end{equation*}
The restriction of $p^G_{\mu,\nu}$ to the 
$\mathcal{Z}(\mathfrak{g})$-primary component of
$H_{\dolbeault(E(\mu))}\otimes F^\nu$ with 
Harish-Chandra parameter $\mu+\rho(\mathfrak{u})+\nu$
is an isomorphism onto the 
$\mathcal{Z}(\mathfrak{g})$-primary component of
$H_{\dolbeault(E(\mu)\otimes E^\nu))}$ with 
Harish-Chandra parameter $\mu+\rho(\mathfrak{u})+\nu$.
\end{theorem}

Let $\Psi$ be the Zuckerman translation functor.
By a theorem of Kostant the smooth representation
$E(\mu)\otimes E^\nu$ of the connected group $H$ is seen
to have a subrepresentation $E(\mu+\nu)$ with generalized
infinitesimal character $\chi_{\mu+\nu}$ and such that
$E(\mu+\nu)=\Psi_{\mu}^{\mu+\nu}(E(\mu))$.
Thanks to theorem \ref{mainthm} we deduce that
$\Psi_{\mu+\rho(\mathfrak{u})}^{\mu+\rho(\mathfrak{u})+\nu}
\left(H_{\dolbeault(E)}\right)
$ is the $\mathcal{Z}(\mathfrak{g})$-primary component of
$H_{\dolbeault(E(\mu))}\otimes F^\nu$ with 
Harish-Chandra parameter $\mu+\rho(\mathfrak{u})+\nu$.
Moreover thanks to theorem \ref{mainthm} again 
(along with condition \eqref{condition} also)
the fact that $H_{\dolbeault(E(\mu+\nu))}$ is the $\mathcal{Z}(\mathfrak{g})$
primary component of $H_{\dolbeault(E(\mu)\otimes E^\nu)}$ with
generalized infinitesimal character $ \chi_{\mu+\rho(\mathfrak{u})+\nu}$.
So theorem \ref{mainthm-3} reads as follows.
\begin{theorem}
Under condition \eqref{condition} the map 
$p_{\mu,\nu}^G\circ\alpha$ induces a $G$-isomorphism
$$
\Psi_{\mu+\rho(\mathfrak{u})}^{\mu+\rho(\mathfrak{u})+\nu}(H_{\dolbeault(E)})
\simeq
H_{\dolbeault(\Psi_{\mu}^{\mu+\nu}(E))}\,.
$$
\end{theorem}
 
This theorem is the geometric analogue of the corresponding theorem on cohomological
induction. A precise statement is given 
in \cite{KV} as theorem 7.237. The proofs we will exlpain here apply
as well if one replaces the space
of sections 
$C^\infty(G,\wedge^\bullet\mathfrak{u}^*\otimes E)^H$ 
by the space 
$\mathrm{Hom}_H(\mathcal{U}(\mathfrak{g})\otimes \wedge^\bullet \mathfrak{u},\,E)$. 
The abstract Dolbeault operator still 
defines on this space a differential complex whose
cohomology is precisely the cohomologically induced
module $\mathcal{R}^\bullet(E)$. So as we claimed in the abstract
we will prove theorem 7.237 in \cite{KV} as well.

The author want to thank  Leticia Barchini, Michel Duflo, Salah Mehdi, Martin Olbrich, Pavle \pandzic and David Vogan
for their interest and useful discussions or comments.

\section{Differential operators on homogeneous spaces}

Let $\sharp \colon \mathcal{U}(\mathfrak{h}) \to \mathcal{U}(\mathfrak{h})$ be the antipode.
This is the antiautomorphism of $\mathcal{U}(\mathfrak{h})$ given
by $X^\sharp=-X$ on $\mathfrak{h}$. Hence for $X_1,\ldots,X_n \in \mathfrak{h}$,
$$ (X_1\cdots X_n)^\sharp = (-1)^nX_n\cdots X_1\,.$$
We consider $\mathcal{U}(\mathfrak{g})$ as a right
$\mathcal{U}(\mathfrak{h})$-module and 
the space $\mathrm{End}(V)$ of continuous linear endomorphisms of 
the smooth $H$-module $V$ 
as the left $\mathcal{U}(\mathfrak{h})$-module given by
$$
(\forall h\in \mathcal{U}(\mathfrak{h}),\, 
T\in \mathrm{End}(V))\qquad 
h\cdot T =T\circ h^\sharp\,.
$$
Let $J$ be the left ideal of $\left(
\mathcal{U}(\mathfrak{g})\otimes \mathrm{End}(V)\right)$ 
generated by elements of the form
$$
 u\cdot h\otimes T - u\otimes h\cdot T
 \qquad
\left(u\in\mathcal{U}(\mathfrak{g})\,
,h\in\mathcal{U}(\mathfrak{h})\,
,T\in\mathrm{End}(V)\right)
$$
Then the amalgamated tensor product over 
$\mathcal{U}(\mathfrak{h})$ is given by
$$ 
\mathcal{U}(\mathfrak{g})
\otimes_{\mathcal{U}(\mathfrak{h})} 
\mathrm{End}(V)
= 
\left(
\mathcal{U}(\mathfrak{g})
\otimes \mathrm{End}(V)
\right)/J
$$
Let $q$ be the quotient map.
The action of $H$ on the tensor product leaves $J$
stable, and hence induces an action of $H$ on the quotient.
The invariant space for this action is the image of the $H$-invariants
in the tensor product under $q$ so that
$$ 
q\colon 
\left(
\mathcal{U}(\mathfrak{g})
\otimes 
\mathrm{End}(V)
\right)^H \to
\left(\mathcal{U}(\mathfrak{g})
\otimes_{\mathcal{U}(\mathfrak{h})} 
\mathrm{End}(V)\right)^H
$$
First note the following well-known lemma.
\begin{lemma} \label{generators}
1. The left $\mathcal{U}(\mathfrak{g})\otimes 
\mathrm{End}(V)$-ideal $J$ is generated by
the element of the form 
$$
Y\otimes I + 1 \otimes \tau(Y)
\qquad (Y\in\mathfrak{h})\,.
$$
2. The module $J^H$ of $H$-invariant elements in $J$ 
is a two-sided ideal of the algebra
$\left(\mathcal{U}(\mathfrak{g})\otimes 
\mathrm{End}(V)\right)^H$.
\end{lemma}
This implies in particular that 
the space 
$
\left(\mathcal{U}(\mathfrak{g})
\otimes_{\mathcal{U}(\mathfrak{h})}
\mathrm{End}(V)\right)^H$
is indeed an algebra.
For $X\in \mathrm{Lie}(G)$ let $r(X)$ be the left 
invariant vector 
field on $G$ given by right differentiation
$$ 
r(X)f(g)=\left[\frac{d}{dt}
f\left(g\exp(tX)\right)\right]_{t=0}\,.
$$
Let us extend $r$ by linearity on $\mathfrak{g}$ and 
on the enveloping algebra $\mathcal{U}(\mathfrak{g})$ by
\begin{equation*}
r(X_1X_2\cdots X_k)=r(X_1)\circ r(X_2)\circ\cdots r(X_k),\;\;\forall X_i\in\mathfrak{g}.
\end{equation*}
One also has right invariant vector field on $G$ defined as
$$ 
l(X)f(g)=\left[\frac{d}{dt}
f\left(\exp(tX)g\right)\right]_{t=0}\,.
$$
and left and right invariant derivatives are
related by
\begin{equation*}
 (l(u)f)(g)=r(\text{Ad}(g)u^\sharp)f)(g)
 \text{ for all }u\in\mathcal{U}(\mathfrak{g}).
\end{equation*}
In particular when $Z\in\mathcal{Z}(\mathfrak{g})$ lies in the center of the
enveloping algebra :
$$ l(Z)=r(Z^\sharp)\,.$$

\begin{proposition}\cite{KoranyiReimann}
The algebra $\mathbb{D}_G(\mathcal{V})$ 
of $G$-invariant differential operators 
on $\mathcal{V}$ is isomorphic
to the algebra 
$
\left(
 \mathcal{U}(\mathfrak{g})
 \otimes_{\mathcal{U}(\mathfrak{h})}
 \mathrm{End}(V)
\right)^H
$. 
This isomorphism
is induced by the following representation
of $\mathcal{U}(\mathfrak{g})\otimes 
\mathrm{End}(V)$ on $C^\infty(G,\,V)$ :
$$
(X\in \mathfrak{g},\, 
T\in\mathrm{End}(V),\,
f\in C^\infty(G,\,V))
\qquad
(X\otimes T)f(g)=T(r(X)f(g))\,.
$$
\end{proposition}
An element
in $(\mathcal{U}(\mathfrak{g})\otimes 
\mathrm{End}(V))^H$ 
isnamed here an abstract differential operator.
If
$D\in (\mathcal{U}(\mathfrak{g})\otimes \mathrm{End}(V))^H$, then 
we define, for any smooth representation $E$ of $H$, an invariant differential
operators on $V\otimes E$
by setting 
$$
D(E)=\varphi_E(D)\,.
$$ 
Note that when $E=E_1\oplus E_2$ as $H$-module 
one has
\begin{equation}\label{exact}
D(E)=D(E_1)\oplus D(E_2)\,.
\end{equation}
Another by-product of this construction is to provide
algebraic operators when a smooth representation $(X,\pi)$ of $G$
is given. Actually we define $D_X$ by
$$ 
\begin{array}{rccl}
 \pi\otimes c \colon &\mathcal{U}(\mathfrak{g})\otimes 
 \mathrm{End}(V)
  &\longrightarrow
  & \mathrm{End}(X\otimes V)\\
 & D & \longmapsto & D_X\,.
\end{array}
$$
Here we note $c(X)$ the action of an element
$X\in\mathrm{End}(V$
on $V$.
If $X$ and $E$ are given it is also useful to consider
the operator 
\begin{equation}\label{algebraic}
 D_X\otimes I_E=\pi\otimes c\otimes I_E (D) \in 
 \mathrm{End}(X\otimes V \otimes E) \,.
\end{equation}

In the context of Dolbeault operators
where $V=\wedge \overline{\mathfrak{u}}$,
we obtain an algebra homomorphism
$$
\varphi_E \colon 
\left(\mathcal{U}(\mathfrak{g})\otimes 
\mathrm{End}(\wedge\overline{\mathfrak{u}})\right)^H 
\xrightarrow{
 \text{~}q\circ (r\otimes c \otimes I_E)\text{~}
}
\left(
\mathcal{U}
(\mathfrak{g})\otimes_{\mathcal{U}
(\mathfrak{h})}
\mathrm{End}(
 \wedge\overline{\mathfrak{u}}\otimes E
)
\right)^H\,.
$$
 The abstract operators
we shall consider are
\begin{gather}
\hat{\dolbeault}=\sum_i X_i\otimes e(\xi_i)\,, \\
v=\sum_{i<j}1\otimes e(\xi_i)e(\xi_j)\iota([X_i,X_j]) \,, \\
\dolbeault = \hat{\dolbeault}+v\,.
\end{gather}

\section{The infinitesimal character}

First we note some basic facts of linear algebra.
If $x$ is a linear operator on a vector space such that $x^2=0$, we note
$ H_x = \frac{\ker x}{\mathrm{im}\,x}$
its cohomology. Let $\mathcal{A}$
a $\mathbb{Z}_2$-graded algebra with grading operator $\gamma$.
The graded commutator $[\,,\,]$ turns $\mathcal{A}$ into a super Lie algebra.
We define $d_a$ as the graded commutator with $a$, that is
$$	d_a(b)=[a,b]\,.$$
This means that if $a \in \mathcal{A}$ is an odd element then
$$d_a(b)=[a,b]=ab-\gamma(b)a\qquad (b\in \mathcal{A})\,.$$ 
If $a$ is even then 
$$d_a(b)=[a,b]=ab-ba\qquad (b\in \mathcal{A})\,.$$
In any case $d_a$ is a graded endomorphism of the $\mathbb{Z}_2$-graded 
vector space $\mathcal{A}$, and has the same degree as $a$ (the space of 
endomorphism is also graded). 
So $a\mapsto d_a$ is an (even) homomorphism
of the underlying super Lie algebras. This means for example
that if $a$ and $b$ are odd elements of $A$, then
\begin{equation}
\label{SuperLieAlgebra}
d_ad_b+d_bd_a=d_{[a,b]}\quad\text{so in particular} \quad  
d_a^2=d_{a^2}=0 \text{ if }a^2=0\,.
\end{equation}
\begin{proposition}
The endomorphism $d_a$ is a graded derivation.
In particular $\ker d_a$ is a subalgebra of 
$\mathrm{End}(\mathcal{A})$. Moreover $\mathrm{im}\,d_a\cap \ker d_a$ 
is a two sided ideal of $\ker d_a$.
In particular, when $a^2=0$, then $H_{d_a}$ is indeed an algebra.
\end{proposition}
\begin{proof}
For $b,c\in \mathcal{A}$ on has for a given odd $a$
$$ d_a(bc)=abc-\gamma(bc)a=
d_a(b)c+\gamma(b)ac-\gamma(b)\gamma(c)a=
d_a(b)c+\gamma(b)d_a(c)\,.$$
Assume $b,c\in \ker d_a$. Then 
it immediately follows that $bc\in \ker d_a$.
Now if moreover $b=d_a(b')$ for some $b'\in \ker d_a^2$,
then, $bc=d_a(b')c+\gamma(b')d_a(c)=d_a(bc)$
and, using $\gamma(c)\in \ker d_a$ (because $\gamma$
anticommutes with $d_a$),
$$cb=cd_a(b')=d_a(\gamma(c)b')-d_a(\gamma(c))b'=d_a(\gamma(c)b')\,.$$
\end{proof}
Now assume we have a graded representation 
of $\mathcal{A}$,
$$ \pi \colon \mathcal{A} \to \mathrm{End}(V)\,,$$
on a $\mathbb{Z}_2$-graded vector space $V$.
Then the representation 
$\pi$ induces by restriction a representation 
of the algebra $\ker d_a$
$$ 
 \pi \colon \ker d_a \to \ker d_{\pi(a)} 
 \subset \mathrm{End}(\ker \pi(a)) \,.
$$
Moreover $\pi(\mathrm{im} \,d_a)\subset
 \mathrm{im} \,d_{\pi(a)}$, and 
\begin{enumerate}
\item 
 The ideal $\mathrm{im} \,d_{\pi(a)}$ 
 sends $\ker \pi(a)$ on
 $\mathrm{im}\,\pi(a)$.
\item
 The algebra $\ker d_{\pi(a)}$ leaves 
 $\mathrm{im}\,\pi(a)$ stable. 
\end{enumerate}
This implies that $\pi$ induces a representation
$$
H_{d_a} \to H_{d_{\pi(a)}} \subset 
\mathrm{End}\left(H_{\pi(a)}\right)\,.
$$

\label{concret}
The main examples we will consider
here are as follows. The space $V$ is the space of smooth sections
of the vector bundle 
$\wedge\overline{\mathcal{U}} \otimes \mathcal{E}$
or 
$\wedge\overline{\mathcal{U}} 
\otimes \mathcal{E} \otimes \mathcal{F}$, 
and when $\mathcal{A}$ is the algebra of abstract 
differential operators, the representation of $\mathcal{A}$
will be $\varphi_E$. 
When we will discuss tensoring with finite dimensional representations $F$ of $G$, 
the algebra will be $\mathbb{D}_G
(\wedge\overline{\mathcal{U}}\otimes\mathcal{E}\otimes\mathcal{F})$
and the operator will be $\Delta_F(\dolbeault(E))$
(see definition below, equation (\ref{delta_F})).
Note that the $\mathbb{Z}$-graduation on a tensor algebra 
induces a $\mathbb{Z}_2$-graduation on the exterior
algebra.

Summarizing this discussion, we have obtained that 
$\varphi_E$ induces a well defined map in cohomology
$$
 \overline{\varphi}_E\colon 
 H_{d_{\dolbeault}}
 \to H_{d_{\dolbeault(E)}}
 \subset \mathrm{End}\left(H_{\dolbeault(E)}\right)\,.
$$ 
The algebra $H_{d_{\dolbeault}}$ 
then acts on $H_{\dolbeault(E)}$. 

\begin{remark}\emph{
The operator $\dolbeault(E)$ is continuous in the smooth topology. Thus its kernel is closed
and inherits the structure of a Fréchet space
with a smooth action of $G$ on it.
Thus the left representation of 
$\mathcal{U}(\mathfrak{g})$ on it is well defined. 
It restricts to $l(Z)$ for $Z\in \mathcal{Z}(\mathfrak{g})$.
If such a $Z$ is seen as an invariant differential
operator acting by $r(Z)\otimes I_{\wedge^\bullet\overline{\mathfrak{u}}\otimes E}$, 
remember that 
$$l(Z)=r(Z^\sharp)\otimes I_{\wedge^\bullet\overline{\mathfrak{u}}\otimes E}\,.$$
}\end{remark}
As any element in $\mathcal{Z}(\mathfrak{g})\otimes  I$
commutes with abstract odd differential operators and has even degree, 
it lies in the kernel of $d_{\dolbeault}$. 
We have proved the following 
\begin{proposition} 
 The restriction to $\mathcal{Z}(\mathfrak{g)}$ 
 of the  (left) action
 of $\mathcal{U}(\mathfrak{g})$ on 
 $\ker \dolbeault(E)$ goes down
 to a well defined action $l$ on $H_{\dolbeault(E)}$.
 Moreover
 if $Z\in\mathcal{Z}(\mathfrak{g)}$, then the action
 of $l(Z)$ 
 on $H_{\dolbeault(E)}$ only
 depends on the class of $Z^\sharp$ in the cohomology 
 space $H_{d_{\dolbeault}}$. 
 More precisely this action is given, 
 for $[f]\in H_{\dolbeault(E)}$ by
$$
 l(Z)[f]=[r\otimes 
  I_{\wedge\overline{\mathfrak{u}}\otimes E} 
  (Z^\sharp) f]\,.
$$
\end{proposition}
 
We now need to compute this representation of $\mathcal{Z}(\mathfrak{g})$ in terms
of Harish-Chandra parameters.
If the operator $\dolbeault$ is replaced by
the cubic Dirac operator $D$ the representation
$\mathcal{Z}(\mathfrak{g})\to H_{d_D}$ has been computed in the proof of the Vogan conjecture as given by Huang and 
\pandzic~\cite{HuangPandzic-book}~\cite{HuangPandzic2002}.
In their proof they determine a homomorphism
$$
 \zeta_D \colon \mathcal{Z}(\mathfrak{g})  
 \rightarrow \mathcal{Z}(\mathfrak{h})
$$
in the case $H=K$ and Kostant~\cite{Kostant2003} extends it to the general case.
Let 
$$
\delta_\mathfrak{h}\colon \mathcal{U}(\mathfrak{h})\to
\mathcal{U}(\mathfrak{g})\otimes 
\mathrm{End(\wedge^\bullet \overline{\mathfrak{u}}})
$$ 
be the derivative of the representation
of $H$ on $\mathcal{U}(\mathfrak{g})\otimes 
\mathrm{End(\wedge^\bullet \overline{\mathfrak{u}}})$.
The key point in the proof of Huang and \pandzic~{is} to get 
a Hodge decomposition
\begin{equation}\label{cohomology}
 \ker d_D\simeq \delta_{\mathfrak{h}} (\mathcal{Z}(\mathfrak{h}))\oplus \mathrm{im}\,d_D\,.
\end{equation}
So $H_{d_D}\simeq\delta_{\mathfrak{h}}(\mathcal{Z}(\mathfrak{h}))
\simeq\mathcal{Z}(\mathfrak{h})$, and the definition of
the map $\zeta_D$ is precisely that $\zeta_D(Z)$ is the unique 
representative of the class of $Z$ in the quotient.
In the next section we will prove the following theorem
which states the analogous result for the Dolbeault operator.
\begin{theorem} \label{cohomology-general}
Let $D=\dolbeault$ be the abstract Dolbeault operator. 
Then the decomposition in equation~(\ref{cohomology})
is still true.
\end{theorem}
In particular we again have
$H_{d_{\dolbeault}}\simeq
\delta_{\mathfrak{h}}(\mathcal{Z}(\mathfrak{h}))
\simeq\mathcal{Z}(\mathfrak{h})$ and we can still
define a map 
$$\zeta_{\dolbeault}\colon
\mathcal{Z}(\mathfrak{g}) \to
\mathcal{Z}(\mathfrak{h})\,.$$
\begin{lemma} 
The map $\zeta_{\dolbeault}$
commutes with the sharp maps $\sharp$ of $\mathfrak{g}$
and $\mathfrak{h}$ ; in other words 
$$
\zeta_{\dolbeault}(Z)^\sharp=
\zeta_{\dolbeault}(Z^\sharp)\,,\qquad (Z\in\mathcal{Z}(\mathfrak{g}))\,.
$$
\end{lemma}
This lemma will be clear from the identification
of the map $\zeta_{\dolbeault}$ given in the 
proof of theorem \ref{cohomology-general} below.
In particular
$r\otimes c \otimes I_E (\zeta_{\dolbeault}(Z^\sharp))
=1\otimes I_{\wedge\overline{\mathfrak{u}}}\otimes \sigma(\zeta_{\dolbeault}(Z^\sharp)^\sharp)
=1\otimes I_{\wedge\overline{\mathfrak{u}}}\otimes \sigma(\zeta_{\dolbeault}(Z))$
on $C^\infty(G,\,\wedge^\bullet\overline{\mathfrak{u}}\otimes E)^H$. Summarizing one gets
$$ 
 l(Z)[f]=
 [r\otimes 
 I_{\wedge\overline{\mathfrak{u}}\otimes E} 
 (Z^\sharp) f]=
 [r\otimes c \otimes I_E (\zeta_{\dolbeault}(Z^\sharp))f]=
 [1\otimes I_{\wedge\overline{\mathfrak{u}}}\otimes\sigma(\zeta_{\dolbeault}(Z))f]\,.
$$ 

\begin{theorem} [theorem \ref{mainthm}]
\label{infinitesimal}
 Let $E$ be a smooth representation
 of $H$ which is $\mathcal{Z}(\mathfrak{h})$-finite. 
 If $ E=\oplus_\mu E(\mu) $
 is a decomposition of $E$ in primary 
 $\mathcal{Z}(\mathfrak{h})$-modules with respective 
 generalized infinitesimal character $\chi_\mu$, 
 then as a 
 $(\mathcal{Z}(\mathfrak{g}),G)$-module 
 $$ 
 H_{\dolbeault(E)} =\oplus_\mu H_{\dolbeault(E(\mu))}\,.
 $$
 and the representation of 
 $\mathcal{Z}(\mathfrak{g})$ on $H_{\dolbeault(E(\mu))}$
 is given by
 the generalized infinitesimal character
 $\chi_{\mu+\rho(\mathfrak{u})}$.
\end{theorem}
\begin{proof} 
 In the case of the Dirac operator, 
 the morphism $\zeta_D$ 
 fits into the following commutative diagram
 \begin{equation}\label{harishchandra}
  \begin{diagram}[height=0.7cm]
   \mathcal{Z}(\mathfrak{g})
   & \rTo^{\zeta_D} & \mathcal{Z}(\mathfrak{h}) \\
   \dTo & & \dTo \\
   S(\mathfrak{t}_\mathfrak{g})^{W_G} & \rTo & S(\mathfrak{t}_\mathfrak{h})^{W_H}
  \end{diagram}
 \end{equation}
 Here the vertical arrows are the 
 Harish-Chandra isomorphism
 and the bottom map is restriction.
 For the Dolbeault operator one needs to know
 what happens to this diagram. This is
 exactly where the $\rho(\mathfrak{u})$-shift appears. 
 The argument is given along the proof of 
 theorem \ref{cohomology-general} in the next section.
\end{proof}
Assume in this theorem that $E$ has a generalized infinitesimal
character with Harish-Chandra parameters $\mu$.
Thanks to
the commutative diagram (\ref{harishchandra}), one has
that $\sigma(\zeta_{\dolbeault} (Z))$, for $Z\in\mathcal{Z}(\mathfrak{g})$ 
acts on $E$ by the infinitesimal character 
$\chi_{\lambda+\rho_{\mathfrak{u}}}$.
We deduce a geometric analogue of the Vogan conjecture
in the context of Dolbeault cohomology. 
\begin{corollary}
 Assume $E=E^\lambda$ is a finite dimensional 
 irreducible 
 representation
 of $H$ with highest weight $\lambda$. Then 
 $\mathcal{Z}(\mathfrak{g})$
 acts on $H_{\dolbeault(E^\lambda)}$ by the infinitesimal character
 $\chi_{\lambda+\rho(\mathfrak{g})}$.
\end{corollary}

\subsection{Proof of theorem \ref{cohomology-general}}
The strategy of the proof follows closely that of Huang 
and \pandzic for the Dirac operator \cite{HuangPandzic2002,HuangPandzic-book}. However
in this case the operator $\dolbeault$ is a differential
complex on $C^\infty(G)\otimes \wedge^\bullet \overline{u}$,
not only the $H$-invariant part. 
We will write $\dolbeault_{\mathrm{full}}$ for 
this extended complex.
So $d_{\dolbeault_{\mathrm{full}}}$ is 
seen as an endomorphism of the algebra 
$\mathcal{U}(\mathfrak{g})\otimes\mathrm{End}(\wedge^\bullet\overline{u})$.
Note that the inclusion of the $H$-invariants
in the whole complex induces a map
$H_{d_{\dolbeault}} \to H_{d_{\dolbeault_{\mathrm{full}}}}$.

The algebra 
$\mathcal{U}(\mathfrak{g})\otimes\mathrm{End}(\wedge^\bullet\overline{u})$
has a natural filtration induced by the filtration 
on $\mathcal{U}(\mathfrak{g})$ and the trivial filtration on 
$\mathrm{End}(\wedge^\bullet\overline{u})\simeq \wedge^\bullet \mathfrak{h}^\perp$.
The associated graded algebra is 
$S(\mathfrak{h})\otimes S(\mathfrak{h}^\perp)\otimes 
\wedge^\bullet\mathfrak{h}^\perp$.
The differential $d_{\dolbeault_{\mathrm{full}}}$ preserves
this filtration (shifting the degree by $1$)
so it goes down to a differential 
$\mathrm{Gr}(d_{\dolbeault_{\mathrm{full}}})$
on the associated graded space by setting 
$
\mathrm{Gr}(d_{\dolbeault_{\mathrm{full}}})(\mathrm{Gr}{a})=
\mathrm{Gr}(d_{\dolbeault_{\mathrm{full}}}(a))\,.
$
The next lemma identifies this differential. 
This computation has been carried out by 
Huang, \pandzic and Renard in 
\cite{HuangPandzicRenard}, remark 2.3. 
Full details can be 
found in \cite[section 3]{HuangPandzicRenard-arxiv}. 
For reader convenience we recall here the 
steps we need from this computation.

Let $\partial_{\mathfrak{u}} \in \mathrm{End}(S(\mathfrak{h}^\perp)\otimes 
\wedge^\bullet\mathfrak{h}^\perp)$ be the Koszul differential 
along $\mathfrak{u}$ :
$$
\partial_{\mathfrak{u}} (u\otimes \omega)= 
\sum_i X_i u\otimes \iota(X_i)\omega\,.
$$
\begin{lemma}\label{commut}
We have the following commutative diagram 
$$
\begin{diagram}
 \mathcal{U}(\mathfrak{g})\otimes
 \mathrm{End}(\wedge^\bullet \overline{\mathfrak{u}})
 & \rTo^{\mathrm{Gr}} &
 \mathcal{S}(\mathfrak{h})\otimes \mathcal{S}(\mathfrak{h}^\perp)\otimes 
 \wedge^\bullet\mathfrak{h}^\perp
\\
\dTo^{d_{\dolbeault_{\mathrm{full}}}} & &
\dTo_{1\otimes\partial_{\mathfrak{u}}}
\\
 \mathcal{U}(\mathfrak{g})
 \otimes\mathrm{End}(\wedge^\bullet \overline{\mathfrak{u}})
 & \rTo^{\mathrm{Gr}} &
 \mathcal{S}(\mathfrak{h})\otimes\mathcal{S}(\mathfrak{h}^\perp)\otimes 
 \wedge^\bullet\mathfrak{h}^\perp
\end{diagram}
$$
\end{lemma}
\begin{proof}
The algebra $\mathrm{End}(\wedge^\bullet\overline{\mathfrak{u}})$ is the algebra
generated by the creation operators $e(\xi_i)$ and annihilation
operators $\iota(X_i)$. These operators satisfy the relations  :
\begin{equation}\label{relation-ext}
\begin{array}{c}
e(\xi_i)e(\xi_j)+e(\xi_j)e(\xi_i)=0\,,\qquad  
\iota(X_i)\iota(X_j)+\iota(X_j)\iota(X_i)=0\vspace{0.2cm}\\
e(\xi_i)\iota(X_j)+\iota(X_j)e(\xi_i)=\delta_{ij}\,.
\end{array}
\end{equation}
So any element of $\mathrm{End}(\wedge\overline{u})$ is in a unique
way a sum of elements of the form
$$ 
w_{IJ}=e(\xi_{i_1})\circ \cdots e(\xi_{i_k})\circ 
\iota(X_{j_1})\circ \cdots \iota(X_{j_l})
\quad
(i_1<\cdots<i_k\,,j_1<\cdots<j_l)\,.
$$
The identification of $\mathrm{End}(\wedge \overline{\mathfrak{u}})$ with $\wedge \mathfrak{h}^\perp$ sends 
an element of this form to 
$$s(w_{IJ})=\xi_{i_1}\wedge\cdots\wedge \xi_{i_k}
\wedge X_{j_1}\wedge\cdots\wedge X_{j_l} \in 
\wedge\overline{\mathfrak{u}}\otimes 
\wedge\mathfrak{u}\simeq \wedge \mathfrak{h}^\perp\,.$$
So for $u\in \mathcal{U}(\mathfrak{g})$, one has
$$
\mathrm{Gr}(d_{\dolbeault})(\mathrm{Gr}(u\otimes w_{IJ}))=
\sum_i X_i\mathrm{Gr}(u)\otimes s([e(\xi_i),w_{IJ}])
$$
where $[e(\xi_i),w_{IJ}]$ is the graded commutator, and thanks to relation~(\ref{relation-ext})
$$
s([e(\xi_i),w_{IJ}]) = \iota(X_i)s(w_{IJ})
$$
\end{proof}
\begin{lemma}\label{koszul}
The inclusion of $S(\mathfrak{h})\otimes
S(\overline{\mathfrak{u}})\otimes \wedge^\bullet \overline{\mathfrak{u}}$ 
in $S(\mathfrak{h})\otimes
S(\mathfrak{h}^\perp)\otimes \wedge^\bullet \mathfrak{h}^\perp$ induces an isomorphism
$
S(\mathfrak{h})\otimes
S(\overline{\mathfrak{u}})\otimes 
\wedge^\bullet \overline{\mathfrak{u}}
\simeq
H_{1\otimes \partial_{u}}
$.
The inclusion of $S(\mathfrak{h})^H\otimes 1
\otimes 1 \otimes 1$ in
$\left(S(\mathfrak{h})\otimes
S(\mathfrak{h}^\perp)\otimes \wedge^\bullet \mathfrak{h}^\perp\right)^H$ 
induces an isomorphism
in cohomology. 
\end{lemma}
This lemma is proved as lemma 3.5 in \cite{HuangPandzicRenard-arxiv}. 
The first part of this lemma is a 
well known fact on Koszul cohomology.
The second part actually follows from the fact that $\mathfrak{h}$ is the 
centralizer of an element $\xi_0$ (defined in the
introduction) such that $\mathrm{ad}\xi_0$ has positive
eigenvalues on $\mathrm{u}$ and negative eigenvalues on 
$\overline{\mathfrak{u}}$.

Now any element of $S(\mathfrak{h})^H\otimes
1\otimes 1$ is the image of an element  in 
$\delta_{\mathfrak{h}}(\mathcal{Z}(\mathfrak{h}))$.
This is true because for $u\in\mathcal{U}(\mathfrak{h})$ one has 
$\mathrm{Gr}(\delta_{\mathfrak{h}}(u))=
\mathrm{Gr}(u\otimes 1\otimes 1)$.
Note that $\delta_{\mathfrak{h}}(\mathcal{Z}(\mathfrak{h}))$
is contained in $\ker d_{\dolbeault}$. Moreover
$\delta_{\mathfrak{h}}(\mathcal{Z}(\mathfrak{h}))\cap 
\mathrm{im} d_{\dolbeault}=0$  because it is 
true on the
right side of the diagram by lemma \ref{koszul}.

Let $a\in \ker d_{\dolbeault}$. We want to show that there exists 
$u\in \mathcal{Z}(\mathfrak{h})$
and $b\in \mathrm{im} d_{\dolbeault}$ such that
$a=\delta_{\mathfrak{h}}(u)+b$. We proceed by
induction. This is obvious if $a$ has order $0$.
Assume this is true for any operator of order less 
than $p-1$, and let
$a\in \mathcal{U}_p\otimes \mathrm{End}
(\wedge^\bullet\overline{\mathfrak{u}})$
such that $\mathrm{Gr}(a)$ has non vanishing 
cohomology class.
Then $\mathrm{Gr}(a)=s+1\otimes \partial_{\mathfrak{u}} (\mathrm{Gr}(b))$ for $s\in S(\mathfrak{h})^H\otimes 1\otimes 1$
and $b\in \mathcal{U}_{p-1}(\mathfrak{g})\otimes \mathrm{End}
(\wedge^\bullet\overline{\mathfrak{u}})$.
Let $u\in \mathcal{Z}(\mathfrak{h})$ such that $\mathrm{Gr}(\delta_{\mathfrak{h}}(u))=s$.
Then 
$$
d_{\dolbeault}\left(a-\delta_\mathfrak{h}(u)-d_{\dolbeault}(b)\right)=0\,.
$$
Moreover $a-\delta_\mathfrak{h}(u)-
d_{\dolbeault}(b)$ has degree $p-1$. 
By assumption, there exist $u'\in \mathcal{Z}(\mathfrak{h})$ and $b'$ such that
$$
a-\delta_\mathfrak{h}(u)-
d_{\dolbeault}(b)=
\delta_\mathfrak{h}(u')+
d_{\dolbeault}(b')
$$
So that $$a-\delta_\mathfrak{h}(u+u')=
d_{\dolbeault}(b+b')$$
So we have proved the theorem \ref{cohomology-general}.

Now we identify the map $\zeta_{\dolbeault}$.
Recall that thanks to the PBW theorem
\begin{equation}\label{decomposition}
\mathcal{U}(\mathfrak{g})=
\mathcal{U}(\mathfrak{h})\oplus \big(\,
\overline{\mathfrak{u}}\,\mathcal{U}(\mathfrak{g})
+\mathcal{U}(\mathfrak{g}){\mathfrak{u}}\,\big)\,.
\end{equation}
and let $p$ be the projection on the first component.
It is a well known fact that if $z\in\mathcal{Z}(\mathfrak{g})$ then
$p(z)\in\mathcal{Z}(\mathfrak{h})$.
Now if $a\in \mathcal{U}(\mathfrak{g})$,
then $$
\mathrm{Gr}(a\otimes 1 \otimes 1)=
\mathrm{Gr}(p(a)\otimes 1 \otimes 1)=
\mathrm{Gr}(\delta_{\mathfrak{h}}(p(a)))\,.
$$
It follows that if $a\in \ker d_{\dolbeault}$ then $a=\delta_{\mathfrak{h}}(p(a)))$ 
in $H_{d_{\dolbeault}}$.

If $X$ is a $\mathcal{U}(\mathfrak{g})$-module with infinitesimal
character $\chi_\mu$ then
the algebra $H_{d_{\dolbeault}}$ has a representation
in $X\otimes \wedge^\bullet\overline{\mathfrak{u}}$
and for $z\in \mathcal{Z}(\mathfrak{g})$ the action
of $z$ is given by multiplication by $\chi_\mu(z)$. 
So $\delta_{\mathfrak{h}}(p(z))$ acts by the same scalar.
We have proved
\begin{lemma}\label{zeta}
One has
\begin{equation}
(\forall z \in \mathcal{Z}(\mathfrak{h}))\quad
\zeta_{\dolbeault}(z)=p(z)
\end{equation}
\end{lemma}

\begin{remark}[Casselman-Osborne theorem]
We have proved that the action of $z$ and $\delta_\mathfrak{h}(p(z))$ on 
$H_{\dolbeault_{X,\text{full}}}$
are equal.
\end{remark}
Let us consider the degree $0$ of $H({\mathfrak{u}},X)$
when $X=F^\lambda$ is an irreducible finite dimensional $G$-module
with highest weight $\lambda=\mu-\rho(\mathfrak{g})$.
We have $H^0(\mathfrak{u},X)=E^\lambda$ and this implies
that for $z\in\mathcal{Z}(\mathfrak{h})$, $\delta_\mathfrak{h}(z)$
acts by the scalar $\chi_{\lambda+\rho(\mathfrak{h})}=
\chi_{\mu-\rho(\mathfrak{u})}$. By a standard density argument
it follows that the map $S(\mathfrak{t})^{W_G}\to S(\mathfrak{t})^{W_H}$ 
defined by prescribing the diagram~(\ref{harishchandra}) to be commutative
is given by $\mu\mapsto \tilde{\mu}=\mu+\rho(\mathfrak{u})$.
So we have identified the map $\zeta_{\dolbeault}$
in terms of the Harish-Chandra isomorphism. We have 
proved theorem \ref{infinitesimal}.

\begin{remark}
This also gives a proof of the following 
weak version of a theorem
of Kostant which computes $H_{\dolbeault_{F^\lambda,\text{full}}}$.
Let 
$W_1=\{w\in W_G\,,\,
(\forall \alpha\in \Delta^+(\mathfrak{g},\mathfrak{t}))
\quad w^{-1}\alpha\in\Delta^+(\mathfrak{h},\mathfrak{t})
\}$.
If $E^\nu$ is an irreducible $H$-module with highest weight $\nu$ such
that $\mathrm{Hom}_H(E^\nu,H_{\dolbeault_{F^\lambda,\text{full}}})\neq 0$
then there exists $w\in W_1$ such that
$$ \nu = w(\lambda+\rho(\mathfrak{g}))-\rho(\mathfrak{g})\,.$$
\end{remark}

\section{The Zuckerman translation functor}

\subsection{
 Tensoring with finite dimensional representation}

Let $(F,\pi)$ (resp. $(V,\tau)$) be a finite 
dimensional (resp. smooth) representation of $G$ (resp. $H$).
\begin{proposition}
The map 
$\alpha \colon C^\infty(G,\, V)^H\otimes F \to
C^\infty(G,\, V\otimes F)^H$ given by
$$ 
 \alpha(f\otimes w)(g) = 
 f(g)\otimes \pi(g)^{-1}w \in  
 V\otimes F
$$
is a smooth $G$-module isomorphism.
\end{proposition}
\begin{proof}
The map $\alpha$ is clearly equivariant and the 
inverse $\beta$ is given as follows.
Let $(w_i)$ be a basis of $F$. For any $g\in G$, 
the family $(\pi(g)w_i)$ is still
a basis of $G$. If 
$f \in C^\infty (G,\, V\otimes F)^H$, then
there exist functions $f_i$ from $G$ to $V$ 
such that for all $g\in G$,
$f(g)=\sum_i \pi(g)^{-1}w_i\otimes f_i(g)$. 
One checks easily that the functions $f_i$
are $H$-invariant. We then set 
$\beta (f) =\sum f_i\otimes w_i$. The map $\beta$
is well defined, equivariant and does not depend 
on the basis $w_i$.
\end{proof}
We now relate the different differential operators on
the bundles in consideration. 
Recall that the algebraic operator 
$D_F\otimes I_E$ has been defined in equation~(\ref{algebraic}).
\begin{proposition}\label{leibniz}
Let $D\in(\mathfrak{h}^\perp\otimes
\mathrm{End}(V)^H$ be an abstract
operator of order $1$ with vanishing order $0$ part.
One has 
$$ 
\alpha \circ (D(E)\otimes I_F) = 
\Big( 
 D(E\otimes F) + 1\otimes (D_F\otimes I_E)
\Big) 
\circ \alpha\,. 
$$
\end{proposition}
\begin{proof}
Let $X\in \mathfrak{h}^\perp$,
$f \in C^\infty (G,\, V\otimes E)$
We have 
\begin{align}
(r(X)\otimes 1)\alpha(f\otimes w)(g) 
&= \left[\frac{d}{dt}
 f(g\exp(tX)) \otimes \pi(\exp(-tX)g^{-1})w\right]_{t=0} \notag \\
&=  r(X)f(g)\otimes g^{-1}w  - f(g)\otimes \pi(X)\pi(g)^{-1}w
 \notag \\
&= \alpha(r(X)f\otimes w)(g) 
  -(1\otimes (I_{\wedge^\bullet\overline{\mathfrak{u}}}\otimes I_E\otimes \pi(X)))
  \alpha(f\otimes w)(g)\label{calculX}
\end{align} 
The proposition follows.
\end{proof}
In our case, this formula reads
$$
\alpha \circ (\dolbeault(E)\otimes I_F)  \circ \beta = 
 \dolbeault(E\otimes F) + 1\otimes (
 \hat{\dolbeault}_F\otimes I_E) =
 \hat{\dolbeault}(E\otimes F) + 1\otimes (
 \dolbeault_F\otimes I_E)
\,.
$$
We will need later an algebraic version of 
proposition \ref{leibniz} that does
not make use of the maps $\alpha$ and $\beta$ in the 
definition of the operator 
$\alpha (D(E)\otimes I) \beta$.
Actually, using  the embedding 
$F\hookrightarrow C^\infty(G,\, F_{|H})^H$ 
(given by $w\mapsto(g\mapsto \rho_w(g)=g^{-1}w)$) 
one sees that $C^\infty(G,\,V\otimes E)^H\otimes F$ 
is a subspace of the space of sections of the 
fibre bundle on $G/H$ obtained by
restriction to the diagonal of $G/H \times G/H$ of 
the bundle 
$(\mathcal{S}\otimes \mathcal{E})\boxtimes\mathcal{F}$. 
So one expects that
the Leibniz rule used in the preceding proposition 
has a formulation in terms
of the coproduct in $\mathcal{U}(\mathfrak{g})$.

Let $\Delta$ be the coproduct in $\mathcal{U}(\mathfrak{g})$
and $(F,\pi)$ a finite dimensional representation of $G$ as before.
We define $\Delta_F= (I\otimes \pi)\circ \Delta$. Hence
\begin{equation} \label{delta_F} 
\Delta_F \colon \mathcal{U}(\mathfrak{g}) \to  
\mathcal{U}(\mathfrak{g})\otimes \mathrm{End}(F)\,.
\end{equation}
\begin{proposition}
The linear map 
$$
\Delta_F\otimes I \colon  
\mathcal{U}(\mathfrak{g})\otimes \mathrm{End}(V\otimes E) \to
\mathcal{U}(\mathfrak{g})\otimes 
\mathrm{End}(V \otimes E\otimes F)
$$
induces an algebra homomorphism
$$ 
{{\Delta}_F} \colon
\mathbb{D}_G(\mathcal{S}\otimes\mathcal{E}) 
\mathop{\longrightarrow}
\mathbb{D}_G(\mathcal{S}\otimes\mathcal{E}\otimes \mathcal{F})
$$
\end{proposition}
\begin{proof}
We have to prove that the map $\Delta_F\otimes 1$ sends the ideal 
$J_{V\otimes E}$ in the definition of 
$\mathcal{U}(\mathfrak{g})
\otimes_{\mathcal{U}(\mathfrak{h})}
\mathrm{End}(V\otimes E)$ to the 
ideal $J_{V\otimes E\otimes F}$ in the definition
of the corresponding quotient.  By lemma \ref{generators} 
it is enough to consider the elements 
$Y\otimes I_{V\otimes E} + 1 \otimes Y$ for $Y\in \mathfrak{h}$.
But 
\begin{align*} 
\Delta_F\otimes I(Y\otimes I_{V\otimes E} &+ 1 \otimes (c\otimes \sigma(Y)))\\
&=
Y\otimes I_{V\otimes E\otimes F} +  
1\otimes \pi(Y)\otimes I_{V\otimes E} + 
1 \otimes (c\otimes \sigma(Y)) \\
&=
Y\otimes I_{V\otimes E\otimes F} +  1\otimes (\pi\otimes c\otimes \sigma (Y))
\end{align*}
It follows that 
$\Delta_F(Y\otimes I_{V\otimes E} + 1 \otimes Y)\in J_{V\otimes E\otimes F}$.
\end{proof}
\begin{proposition}
One has ${\Delta}_F(D(E))=\alpha(D(E)\otimes I)\beta$
\end{proposition}
\begin{proof}
Let us compute ${\Delta}_F\otimes I(X\otimes 
I_{V\otimes E\otimes F})$ on $\alpha(f\otimes w)$
for $X\in \mathfrak{h}^\perp$.
We have
\begin{align*}
{\Delta}_F\otimes I(X\otimes 
&I_{V\otimes E\otimes F})\alpha(f\otimes w)\\
&= (\Delta_F(X)\otimes I)\alpha(f\otimes w) \\
&=(r(X)\otimes I_{V\otimes E\otimes F}
   +1\otimes I_{V\otimes E}\otimes \pi(X))\alpha(f\otimes w) \\
&=-(1\otimes I_{V \otimes E}\otimes \pi(X))\alpha(f\otimes w)
   +\alpha((r(X)f\otimes w)\\
&\phantom{=-(1\otimes I_{V \otimes E}\otimes \pi(X))\alpha(f\otimes w)+\alpha}
 +(1\otimes I_{V\otimes E}\otimes \pi(X))\alpha (f\otimes w)\\
&=\alpha((r(X))f\otimes w)
\end{align*} 
by the computation~(\ref{calculX}).
The proposition follows.
\end{proof}
Summarising this discussion we have obtained 
the following result.
\begin{lemma} 
Let $D$ be an abstract differential operator
of order $1$ with vanishing order $0$ part.
Then for any  finite dimensional representation 
$(F,\pi)$
of $G$ and smooth representation $E$ of $H$
the operators ${\Delta}_F(D(E))$, $ D(E\otimes F)$
and $D_F\otimes I_E$ are related by the following formula :
$$
{\Delta}_F(D(E))=D(E\otimes F)+1\otimes
(D_F\otimes I_E) \,.
$$
\end{lemma}
As before we then have
$$
{\Delta}_F(\dolbeault(E))=
\dolbeault(E\otimes F) + 1\otimes (
 \hat{\dolbeault}_F\otimes I_E) =
 \hat{\dolbeault}(E\otimes F) + 1\otimes (
 \dolbeault_F\otimes I_E)
$$
\subsection{Proof of theorem \ref{mainthm-2}}
We now prove theorem \ref{mainthm-2}.
Barchini proposed it as exercise (b) of lecture 2 in 
\cite{barchini:ias:park-city} but it is perhaps not
as obvious as expected. The isomorphism $\alpha$ satisfies
$$ 
\Delta_F(\dolbeault(E))\circ \alpha = 
\alpha \circ (\dolbeault(E)\otimes I_F)\,.
$$
In particular, it induces an isomorphism at the level of cohomology.
\begin{proposition}\label{alpha}
The linear map $\alpha$ induces a $G$-module
isomorphism
\begin{equation*}
{\alpha}\colon H_{{\dolbeault}(E)}\otimes F 
\mathop{\longrightarrow}\limits^\sim
H_{\Delta_F({\dolbeault}(E))}\,,
\end{equation*}
whose inverse is induced by $\beta=\alpha^{-1}$.
\end{proposition}
 
On the other side, thanks to equation~(\ref{exact}),
the map $p^H_{\mu,\nu}$ induces
$$ 
p^G_{\mu,\nu}\colon
H_{\dolbeault(E(\mu)\otimes F^\nu)}
\longrightarrow
H_{\dolbeault(E((\mu)\otimes E^\nu))}
$$
But a priori
$
H_{\Delta_{F^\nu}(\dolbeault(E(\mu)))}
$
and
$
H_{\dolbeault(E(\mu)\otimes F^\nu)}
$
may be different. 

The restriction to $H$ of the representation $F^\nu$
contains a finite dimensional irreducible representation $E^\nu$
with highest weight $\nu$. So as an $H$-module, one has
$F^\nu=E^\nu\oplus E'$.
Let us consider the bundle 
$\wedge^\bullet \overline{\mathfrak{u}} \otimes \mathcal{E}(\mu)\otimes \mathcal{F}^\nu$.
We decompose it as follows
\begin{equation}\label{decomp}
\wedge^\bullet \overline{\mathfrak{u}} \otimes \mathcal{E}(\mu)\otimes \mathcal{F}^\nu=
\left(
 \wedge^\bullet \overline{\mathfrak{u}} \otimes \mathcal{E}(\mu)\otimes \mathcal{E}^\nu
\right)
\oplus 
\left(
 \wedge^\bullet \overline{\mathfrak{u}} \otimes \mathcal{E}(\mu)\otimes \mathcal{E}'
\right)\,.
\end{equation}
Note that 
$\wedge^\bullet \overline{\mathfrak{u}} \otimes \mathcal{E}({\mu+\nu}) \subset
\wedge^\bullet \overline{\mathfrak{u}} \otimes \mathcal{E}(\mu)\otimes \mathcal{E}^\nu$ and that $\nu$ is not a weight of $E'$.

\begin{proposition}
In the decomposition~(\ref{decomp}) the operator
$1\otimes \hat{\dolbeault}_{F^\nu} \otimes I_{E(\mu)}$
has the form
$$
1\otimes \hat{\dolbeault}_{F^\nu} \otimes I_{E(\mu)}
= \left(\begin{array}{cc}0&0\\ *&*\end{array}\right)\,.
$$
In other words the range of $1\otimes \hat{\dolbeault}_{F^\nu} \otimes I_{E(\mu)}$
is contained in 
$\Gamma(\wedge^\bullet \overline{\mathfrak{u}} \otimes E(\mu)\otimes E')$.
\end{proposition}
\begin{proof}
The range of $\hat{\dolbeault}_{F^\nu}$ is
contained in $\wedge \overline{\mathfrak{u}}\otimes \overline{\mathfrak{u}}F^\nu$. But
$F^\nu/\overline{\mathfrak{u}}F^\nu$ is a highest
weight module for $H$ with highest weight $\nu$.
So $E^\nu\subset F^\nu/\overline{\mathfrak{u}}F^\nu$ and $E' \supset \overline{\mathfrak{u}}F^\nu$.
\end{proof}
Thanks to this proposition, we see that 
$$
p^G_{\mu,\nu}\circ \Delta_{F^\nu}(\dolbeault(E(\mu)))=
\dolbeault(E(\mu)\otimes E^\nu))\circ p^G_{\mu,\nu}\,.
$$
so that $p^G_{\mu,\nu}$ goes down to a $G$-map on cohomology.
Together with proposition~\ref{alpha}, this implies theorem~\ref{mainthm-2}.
\begin{remark}\emph{
This is a fundamental difference with the Dirac
operator. In this case we only have
$
1\otimes D_{F^\nu} \otimes I_{E^\lambda}
= \left(\begin{smallmatrix}0&*\\ *&*\end{smallmatrix}\right)
$
so $p^H_{\mu,\nu}$ does not induces a map 
from $\ker(\Delta_{F^\nu}(D(E(\mu))))$ to $\ker(D(E(\mu)\otimes E^\nu)))$.}
\end{remark}
Also note that the subspace 
$\left(C^\infty(G)\otimes\wedge^\bullet\overline{\mathfrak{u}}\otimes \otimes E(\mu) \otimes E'\right)^H$
is stable by the operator $\Delta_{F^\nu}(\dolbeault(E(\mu)))$. In particular
the inclusion $\iota \colon E'\to F^\nu$ yields a long exact sequence
in cohomology, together with an inclusion of $G$-maps
\begin{equation}\label{factor}
\frac{
	H_{\Delta_{F^\nu}(\dolbeault(E(\mu)))}
	}
	{
	\iota
	\left(
		H_{\Delta_{F^\nu}(\dolbeault(E(\mu)))
			|_{\Gamma(\wedge^\bullet \overline{\mathfrak{u}} 
		  	\otimes E(\mu) \otimes E')}} 
	\right)
	}	 
\longrightarrow
H_{\dolbeault(E(\mu)\otimes E^\nu)}
\end{equation}
whose range is the kernel of the ($G$-equivariant) connecting map
\begin{equation}\label{connecting}
H_{\dolbeault(E(\mu)\otimes E^\nu))}
\longrightarrow
H^{+1}_{\Delta_{F^\nu}(\dolbeault(E(\mu)))
	|_{\Gamma(\wedge^\bullet \overline{\mathfrak{u}} 
	   \otimes E(\mu) \otimes E')
	  }
  } 
\end{equation}
We are now going to prove that, under some equiregularity condition,
the left hand side of equation~\eqref{factor}
contains to the full $\mathcal{Z}(\mathfrak{g})$ primary component
of $H_{\Delta_{F^\nu}(\dolbeault(E(\mu)))}$ with generalized
infinitesimal character $\chi_{\mu+\rho(\mathfrak{u})+\nu}$ and the connecting
map~\eqref{connecting} vanishes.

\subsection{Proof of theorem \ref{mainthm-3}}
As explained in the preceding section the algebra
$H_{d_{\Delta_F(\dolbeault(E))}}$ acts on the space 
$H_{\Delta_F(\dolbeault(E))}$.
Thus it makes sense to consider the class of $Z\otimes I_{\wedge^\bullet\overline{\mathfrak{u}}\otimes E\otimes F}$ in this
algebra, for $Z\in\mathcal{Z}(\mathfrak{g})$, and the action of 
$Z\otimes I_{\wedge^\bullet\overline{\mathfrak{u}}\otimes E\otimes F}$ on
$H_{\Delta_F(\dolbeault(E))}\simeq H_{\dolbeault(E)}\otimes F$ only depends
on its class in  $H_{d_{\Delta_F(\dolbeault(E))}}$.
Nevertheless the algebra 
$H_{d_{\Delta_F(\dolbeault(E))}}$
is not the image of $H_{d_{\dolbeault}}$ 
by $\overline{\varphi}_{E\otimes F}$ 
so that the characters
appearing in this action of $\mathcal{Z}(\mathfrak{g})$
may not be computed using the arguments 
of the previous section.
So we have to compute $H_{d_{\Delta_F(\dolbeault(E))}}$
and describe the homomorphism 
$\mathcal{Z}(\mathfrak{h})\to H_{d_{\Delta_F(\dolbeault(E))}}
\to\mathrm{End}\left(H_{{\Delta_F(\dolbeault(E))}}\right)$
explicitly. Actually we will need to compute
\begin{equation}\label{tocompute}
\mathcal{Z}(\mathfrak{g}) 
\to 
H_{d_{\Delta_{F^\nu}(\dolbeault(E(\mu)))
	|_{\Gamma(\wedge^\bullet \overline{\mathfrak{u}} 
	   \otimes E(\mu) \otimes E')}}}
\to
\mathrm{End}
	\left( 
		H_{{\Delta_{F^\nu}(\dolbeault(E(\mu)))
	  		|_{\Gamma(\wedge^\bullet \overline{\mathfrak{u}} 
			   \otimes E(\mu) \otimes E')}}}
	\right)	
\end{equation}
also. Here is the main technical result we need.
We will prove it in the next paragraph.

\begin{lemma} \label{mainlemma}
Let $E$ be a smooth representation of $H$ with finite length and $F$ a 
finite dimensional representation of $G$. 
Let $E'$ is a $(\mathfrak{q},H)$-submodule of $E\otimes F$ so that
$\Gamma(\wedge^\bullet \mathfrak{u}\otimes E')$ 
is stable by the operator  $\Delta_F(\dolbeault(E))$. 
Then the action of $z\in\mathcal{Z}(\mathfrak{g})$ on
$H_{{\Delta_{F^\nu}(\dolbeault(E(\mu)))
	  		|_{\Gamma(\wedge^\bullet \overline{\mathfrak{u}} 
			   \otimes E(\mu) \otimes E')}}}$
is given by the action of $p(z)$ on 
$$
H(\mathfrak{u},E')\subset H(\mathfrak{u},F)\otimes E
\subset \wedge^\bullet \overline{\mathfrak{u}}\otimes E \otimes F\,.
$$
\end{lemma}
Let $F^\nu $ is an irreducible finite dimensional representation
of $G$ with highest weight $\nu$. If $E$ is a smooth representation of $H$ 
we note $E(\mu)$ its $\mathcal{Z}(\mathfrak{g})$ primary component
generalized infinitesimal character $\chi_\mu$.
The $\mathcal{U}(\mathfrak{g})$-module $H_{\dolbeault(E(\mu))}\otimes F^\nu$
splits into a direct sum
$$
H_{\dolbeault(E(\mu))}\otimes F^\nu = \oplus_{\nu'}H_{\mu+\rho(\mathfrak{u})+\nu'}\,,
$$
where $\nu'$ is a weight of $F^\nu$, and $H_{\mu+\rho(\mathfrak{u})+\nu'}$
is the $\mathcal{Z}(\mathfrak{g})$ primary component with generalized
infinitesimal characters $\chi_{\mu+\rho(\mathfrak{u})+\nu'}\,.$ 

Assume that 
\begin{equation}\tag{C}\label{condition}
\mu + \rho(\mathfrak{u})+\nu 
\text{ is at least singular as }
\mu + \rho(\mathfrak{u})\,.
\end{equation}
Under this condition, thanks to lemma \eqref{mainlemma} (applied to
$E=E(\mu)$ and $F=F^\nu$) and 
\cite[Proposition 7.166]{KV} the map $\iota$ in equation $\eqref{factor}$
is injective and the connecting map \eqref{connecting} vanishes.
So we have an isomorphism of $G$-representation
\begin{equation}
\frac{
	H_{\Delta_{F^\nu}(\dolbeault(E(\mu)))}
	}
	{
	H_{\Delta_{F^\nu}(\dolbeault(E(\mu)))
			|_{\Gamma(\wedge^\bullet \overline{\mathfrak{u}} 
		  	\otimes E(\mu) \otimes E')}} 
	}	 
\mathop{\longrightarrow}\limits^\sim
H_{\dolbeault(E(\mu)\otimes E^\nu)}
\end{equation}
and the projection on the component $H_{\mu+\rho(\mathfrak{u})+\nu}$ in
$H_{\Delta_{F^\nu}(\dolbeault(E(\mu)))}\simeq H_{\dolbeault(E(\mu))}\otimes F^\nu$
factors through the left side of this isomorphism.
It follows that 
\begin{theorem}[Theorem \ref{mainthm-3}]\label{explicit}
Under condition \eqref{condition} the map 
$p_{\mu,\nu}^G\circ\alpha$ induces is a surjective map in cohomology
$$
p_{\mu,\nu}^G\circ\alpha \colon
H_{\dolbeault(E(\mu))}\otimes F^\nu \to
H_{\dolbeault(E(\mu)\otimes E^\nu)}
$$
that induces an isomorphism onto the $\mathcal{Z}(\mathfrak{g})$
primary components with
generalized infinitesimal character $ \chi_{\mu+\rho(\mathfrak{u})+\nu}$
on both sides.
\end{theorem}

\subsection{Proof of lemma \ref{mainlemma}}

In terms of enveloping and symmetric Lie algebras
the principal symbol map is
$$ 
\mathrm{Gr}_E  \colon 
\mathbb{D}_G(\wedge^\bullet \overline{\mathfrak{u}}\otimes E \otimes F)
\to
\left(
 \mathcal{S}(\mathfrak{h}^\perp)\otimes 
 \wedge\mathfrak{h}^\perp\otimes \mathrm{End}(E)
\right)^H
$$

Let 
$\mathbb{D}_G^p(\wedge^\bullet \overline{\mathfrak{u}}\otimes E \otimes F)$
be the space of invariant differential operators with order less than $p$.
\begin{proposition}
Let $d=d_{\Delta_F(\dolbeault(E))}$ denotes the 
graded commutator with $\Delta_F(\dolbeault(E))$,
and $\partial=\partial_{\mathfrak{u}}$ 
the Koszul differential 
along $\mathfrak{u}$ on the 
algebra $\mathcal{S}(\mathfrak{h}^\perp)\otimes 
 \wedge\mathfrak{h}^\perp$ as before.

We have a commutative diagram
$$
\begin{diagram}
\mathbb{D}_G^p(\wedge^\bullet \overline{\mathfrak{u}}\otimes E \otimes F)
& \rTo^{\mathrm{Gr}_{E\otimes F}^{p}} &
\left(
 \mathcal{S}_{p}(\mathfrak{h}^\perp)\otimes 
 \wedge 
 \mathfrak{h}^\perp\otimes
 \mathrm{End}(E\otimes F)
\right)^H \\
\dTo_{d} & &\dTo_{\partial\otimes I_{E\otimes F}}
\\
\mathbb{D}_G^{p+1}(\wedge^\bullet \overline{\mathfrak{u}}\otimes E \otimes F)
& \rTo^{\mathrm{Gr}_{E\otimes F}^{p+1}} &
\left(
 \mathcal{S}_{p+1}(\mathfrak{h}^\perp)\otimes 
 \wedge 
 \mathfrak{h}^\perp\otimes
 \mathrm{End}(E\otimes F)
\right)^H
\end{diagram}
$$
\end{proposition} 
The proof is formally identical to that 
of lemma \ref{commut}. 
In fact, only the principal symbol of $\Delta_F(\dolbeault(E))$, namely
$\hat{\dolbeault}(E\otimes F)$, contributes to $\partial$.
 
According to this proposition the snake lemma now 
provides as usual a long exact sequence in cohomology
$$
H^0_d \mathop{\longrightarrow}\limits^{\mathrm{Gr}} H^0_\partial
\mathop{\longrightarrow}\limits^S H^{0}_{d}\to H_d^1 \to \cdots
\to H^{p-1}_{\partial}
\mathop{\longrightarrow}\limits^S H^{p-1}_{d}\to H^p_{d} 
 \xrightarrow{\mathrm{Gr}_{E\otimes F}^p}
H^p_{\partial}\mathop{\longrightarrow}\limits^S H_{d}^{p}
\to \cdots
$$
Moreover thanks to the decomposition 
$$
\Delta_F(\dolbeault(E))=\hat{\dolbeault}(E\otimes F)+1\otimes \dolbeault_{F,\mathrm{full}}\otimes I_E
$$
one sees that in the associated spectral sequence,
the operator $d_1$ is given for $[\sigma_P]=[\mathrm{Gr}(P)]\in H_\partial=E_1$ by
$$
d_1[\sigma_P]=[\mathrm{Gr}(d_{1\otimes\dolbeault_F\otimes I_E}P)]\,.
$$

Now the cohomology spaces $H_\partial$ are
given by the following lemma whose proof 
is similar to that of lemma \ref{koszul}.

\begin{lemma}
	One has 
	$$
	H_{\partial}=
	\left(
	 S(\overline{\mathfrak{u}}) \otimes \wedge^\bullet\overline{\mathfrak{u}}
	 \otimes \mathrm{End}(E\otimes F)
	\right)^H
	\subset
	\left(
	 S(\overline{\mathfrak{u}})\otimes 
	 \mathrm{End}(\wedge^\bullet\overline{\mathfrak{u}}\otimes E\otimes F)
	\right)^H
	$$
\end{lemma}
In particular, considering the action of the element $\xi_0$ in the center 
of $\mathfrak{h}$ one sees that 
$H_{\partial}$ is finite dimensional, so the spectral 
sequence converges to $H_{d_{\Delta_F(\dolbeault(E))}}$.
Moreover if $Z\in \mathcal{Z}(\mathfrak{g})$, then by 
\cite[Lemma 4.123]{KV} we may write
$Z=p(Z)+uX$ with $u\in\mathcal{U}(\mathfrak{g})$ and $X\in\mathfrak{u}$.
It follows that the images of $\mathrm{Gr}(Z-p(Z))\in 
\mathcal{U}(\mathfrak{g})\mathfrak{u}\otimes 1$ so  
$\mathrm{Gr}(Z-p(Z))= d_0 \omega_0$
with $\omega_0 \in (S(\mathfrak{g})\mathfrak{u}\otimes 1)
\oplus(S(\mathfrak{g})\otimes (\oplus_{k>0}\wedge^k\mathfrak{u}))$.
Proceeding by induction on constructs a sequence
$\omega_r$ of operators with decreasing ordrer
such that $Z-p(Z)-\sum_{k<r}\omega_{k}=d_r(\omega_r)$.
It follows that $Z=p(Z)$ in $H_{d_{\Delta_F(\dolbeault(E))}}$.
Hence if $Z\in\mathcal{Z}(\mathfrak{g})$ its
representative in $H_{d_{\Delta_F(\dolbeault(E))}}$
is given by
$c\otimes\sigma\otimes\pi(p(Z))\in\mathrm{End}_H(\wedge^\bullet 
\overline{\mathfrak{u}} \otimes E\otimes F)$.
We then have in $H_{d_{\Delta_F(\dolbeault(E))}}$
\begin{align*}
l(Z)&=r(Z)^\sharp\otimes I_{\wedge^\bullet\overline{\mathfrak{u}}\otimes E \otimes F} \\
&=1\otimes c\otimes\sigma\otimes \pi(p(Z)) \,.
\end{align*}


\bibliographystyle{alpha}

\begin{thebibliography}{HPR06}

\bibitem[AM00]{AM00}
A.~Alekseev and E.~Meinrenken.
\newblock The non-commutative {W}eil algebra.
\newblock {\em Invent. Math.}, 139(1):135--172, 2000.

\bibitem[AM05]{AM05}
Anton Alekseev and Eckhard Meinrenken.
\newblock Lie theory and the {C}hern-{W}eil homomorphism.
\newblock {\em Ann. Sci. \'Ecole Norm. Sup. (4)}, 38(2):303--338, 2005.

\bibitem[Bar00]{barchini:ias:park-city}
Leticia Barchini.
\newblock Unitary representations attached to elliptic orbits. {A} geometric
  approach.
\newblock In {\em Representation theory of Lie groups (Park City, UT, 1998)},
  volume~8 of {\em IAS/Park City Math. Ser.}, pages 149--176. Amer. Math. Soc.,
  Providence, RI, 2000.

\bibitem[Duf83]{duflo}
Michel Duflo.
\newblock Opérateurs différentiels invariants et homologie des algèbres de
  lie, {\em appendice à un cours sur les opérateurs différentiels
  invriants}.
\newblock {\em Private communication}, pages 1--22, 1983.

\bibitem[HP02]{HuangPandzic2002}
Jing-Song Huang and Pavle Pand{\v{z}}i{\'c}.
\newblock Dirac cohomology, unitary representations and a proof of a conjecture
  of {V}ogan.
\newblock {\em J. Amer. Math. Soc.}, 15(1):185--202, 2002.

\bibitem[HP06]{HuangPandzic-book}
Jing-Song Huang and Pavle Pand{\v{z}}i{\'c}.
\newblock {\em Dirac operators in representation theory}.
\newblock Mathematics: Theory \& Applications. Birkh\"auser Boston Inc.,
  Boston, MA, 2006.

\bibitem[HPR05]{HuangPandzicRenard-arxiv}
Jing-Song Huang, Pavle Pand{\v{z}}i{\'c}, and David Renard.
\newblock Dirac operators and {L}ie algebra cohomology.
\newblock {\em arXiv:0503.582}, 2005.

\bibitem[HPR06]{HuangPandzicRenard}
Jing-Song Huang, Pavle Pand{\v{z}}i{\'c}, and David Renard.
\newblock Dirac operators and {L}ie algebra cohomology.
\newblock {\em Represent. Theory}, 10:299--313 (electronic), 2006.

\bibitem[Kos03]{Kostant2003}
Bertram Kostant.
\newblock Dirac cohomology for the cubic {D}irac operator.
\newblock In {\em Studies in memory of {I}ssai {S}chur ({C}hevaleret/{R}ehovot,
  2000)}, volume 210 of {\em Progr. Math.}, pages 69--93. Birkh\"auser Boston,
  Boston, MA, 2003.

\bibitem[KR00]{KoranyiReimann}
A.~Kor{\'a}nyi and H.~M. Reimann.
\newblock Equivariant first order differential operators on boundaries of
  symmetric spaces.
\newblock {\em Invent. Math.}, 139(2):371--390, 2000.

\bibitem[KV95]{KV}
Anthony~W. Knapp and David~A. Vogan, Jr.
\newblock {\em Cohomological induction and unitary representations}, volume~45
  of {\em Princeton Mathematical Series}.
\newblock Princeton University Press, Princeton, NJ, 1995.

\bibitem[Wol69]{wolf}
Joseph~A. Wolf.
\newblock The action of a real semisimple group on a complex flag manifold.
  {I}. {O}rbit structure and holomorphic arc components.
\newblock {\em Bull. Amer. Math. Soc.}, 75:1121--1237, 1969.

\bibitem[Won95]{wong:95}
Hon-Wai Wong.
\newblock Dolbeault cohomological realization of {Z}uckerman modules associated
  with finite rank representations.
\newblock {\em J. Funct. Anal.}, 129(2):428--454, 1995.

\bibitem[Won99]{wong:99}
Hon-Wai Wong.
\newblock Cohomological induction in various categories and the maximal
  globalization conjecture.
\newblock {\em Duke Math. J.}, 96(1):1--27, 1999.

\end{thebibliography}

\end{document}